\documentclass[reqno]{amsart}
\usepackage{amssymb}
\usepackage{hyperref}

\begin{document}
\title[ The $(p,q)$-Analogues of Some Inequalities for the Digamma Function]
{The $(p,q)$-Analogues of Some Inequalities for the Digamma Function}

\author[Kwara Nantomah]
{Kwara Nantomah}

\address{Kwara Nantomah \newline
Department of Mathematics, University for Development Studies, Navrongo Campus, P. O. Box 24, Navrongo, UE/R, Ghana.}
\email{mykwarasoft@yahoo.com, knantomah@uds.edu.gh}


\begin{abstract}
In this paper, we present the $(p,q)$-analogues of some inequalities concerning the digamma function. Our results generalize some earlier results.
\end{abstract}

\maketitle
\numberwithin{equation}{section}
\newtheorem{theorem}{Theorem}[section]
\newtheorem{lemma}[theorem]{Lemma}
\newtheorem{proposition}[theorem]{Proposition}
\newtheorem{corollary}[theorem]{Corollary}
\newtheorem*{remark}{Remark}

\section{Introduction and Preliminaries}
\noindent
The classical Euler's Gamma function, $\Gamma(t)$ and the digamma function, $\psi(t)$ are commonly defined as
\begin{equation*}\label{eqn:gamma-and-digamma}
\Gamma(t)=\int_0^\infty e^{-x}x^{t-1}\,dx  \quad \text{and} \quad \psi(t)=\frac{d}{dt}\ln \Gamma(t) =\frac{\Gamma'(t)}{\Gamma(t)}, \quad t>0.
\end{equation*}

\noindent
The $p$-analogues of the Gamma and digamma functions are respectively defined as follows. 
\begin{equation*}\label{eqn:p-gamma-and-p-digamma}
\Gamma_p(t)=\frac{p!p^t}{t(t+1) \dots (t+p)} \quad \text{and} \quad \psi_p(t)=\frac{d}{dt}\ln \Gamma_p(t) =\frac{\Gamma_p'(t)}{\Gamma_p(t)}, \quad t>0.
\end{equation*}
where \,  $\lim_{p\rightarrow \infty}\Gamma_p(t)=\Gamma(t)$ \,  and  \,   $\lim_{p\rightarrow  \infty}\psi_p(t)=\psi(t)$. For some more insights and properties of these functions, see  \cite{Krasniqi-Mansour-Shabani-2010},  \cite{Krasniqi-Shabani-2010} and the references therein.\\

\noindent
Similarly,  the $q$-analogues of the Gamma and digamma  functions are respectively defined for $q\in(0,1)$ as (see also   \cite{Krasniqi-Mansour-Shabani-2010} and  \cite{Krasniqi-Shabani-2010})
\begin{equation*}\label{eqn:q-gamma-and-q-digamma}
\Gamma_q(t)=(1-q)^{1-t}\prod_{n=1}^{\infty}\frac{1-q^n}{1-q^{t+n}} \quad \text{and} \quad \psi_q(t)=\frac{d}{dt}\ln \Gamma_q(t) =\frac{\Gamma_q'(t)}{\Gamma_q(t)}, \quad t>0.
\end{equation*}
where \,  $\lim_{q\rightarrow 1^-}\Gamma_q(t)=\Gamma(t)$ \,  and  \,   $\lim_{q\rightarrow 1^-}\psi_q(t)=\psi(t)$.\\

\noindent
In 2012, Krasniqi and Merovci \cite{Krasniqu-Merovci-2012} defined the $(p,q)$-analogue of the Gamma function, $\Gamma_{p,q}(t)$ as
\begin{equation*}\label{eqn:(p,q)-gamma}
\Gamma_{p,q}(t)=\frac{[p]_{q}^{t}[p]_{q}!}{[t]_{q}[t+1]_{q}\dots[t+p]_{q}} ,\quad t>0, \quad p\in N, \quad q\in(0,1).
\end{equation*}
where \, $[p]_{q}=\frac{1-q^p}{1-q}$.\,  For several properties and characteristics of this function, we refer to \cite{Krasniqi-Srivastava-Dragomir-2013} \\

\noindent
Similarly, the $(p,q)$-analogue of the digamma function $\psi_{p,q}(t)$  is defined as
\begin{equation*}\label{eqn:(p,q)-digamma}
\psi_{p,q}(t)=\frac{d}{dt}\ln \Gamma_{p,q}(t) =\frac{\Gamma'_{p,q}(t)}{\Gamma_{p,q}(t)}, \quad t>0, \quad p\in N, \quad q\in(0,1).
\end{equation*}

\noindent
The functions $\psi(t)$ and  $\psi_{p,q}(t)$  as defined above have the following series representations.
\begin{align*}
\psi(t) &=-\gamma + (t-1) \sum_{n=0}^{\infty}\frac{1}{(1+n)(n+t)},\quad t>0\\
\psi_{p,q}(t)&= \ln[p]_q  + (\ln q)\sum_{n=1}^{p}\frac{q^{nt}}{1-q^{n}}, \qquad t>0. 
\end{align*}
where $\gamma$ is the Euler-Mascheroni's constant.\\

\noindent
By taking the $m$-th derivative of these functions, it can easily be shown that the following statements are valid for $m\in N$.

\begin{align*}
\psi^{(m)}(t) &= (-1)^{m+1}m! \sum_{n=0}^{\infty}\frac{1}{(n+t)^{m+1}},\quad t>0 \\
\psi_{p,q}^{(m)}(t)&= (\ln q)^{m+1}\sum_{n=1}^{p}\frac{n^{m}q^{nt}}{1-q^{n}},\quad t>0. 
\end{align*}

\noindent
In 2011, Sulaiman \cite{Sulaiman-2011}  presented the following results.
\begin{equation}\label{eqn:digamma-ineq}
\psi(s+t)\geq \psi(s) + \psi(t) 
\end{equation}
for $t>0$ and  $0<s<1$.

\begin{equation}\label{eqn:polygamma-ineq-1}
\psi^{(m)}(s+t)\leq \psi^{(m)}(s) + \psi^{(m)}(t) 
\end{equation}
for $s,t>0$ and for a positive odd integer $m$.

\begin{equation}\label{eqn:polygamma-ineq-2}
\psi^{(m)}(s+t)\geq \psi^{(m)}(s) + \psi^{(m)}(t) 
\end{equation}
for $s,t>0$ and for a positive even integer $m$.

\begin{equation}\label{eqn:polygamma-ineq-3}
\psi^{(m)}(s)\psi^{(m)}(t)\geq \left[ \psi^{(m)}(s+t) \right]^2
\end{equation}
for $s,t>0$ and for a positive odd integer $m$.\\

\noindent
Prior to Sulaiman's results, Mansour and Shabani  by using different techniques established similar  inequalities for the function $\psi_q(t)$. These can be found in \cite{Mansour-Shabani-2009}.\\

\noindent
Our objective in this paper is to establish that the inequalities  ~(\ref{eqn:digamma-ineq}), ~(\ref{eqn:polygamma-ineq-1}), ~(\ref{eqn:polygamma-ineq-2}) and ~(\ref{eqn:polygamma-ineq-3}) still hold true for the $(p,q)$-analogue of the digamma function. 


\section{Main Results}
\noindent
We now present the results of this paper.


\begin{theorem}\label{thm:(p,q)-digamma-ineq-1}
Let $t>0$, $0<s\leq1$,  $ q\in(0,1)$ and $p\in N$. Then the following inequality is valid.
\begin{equation}\label{eqn:(p,q)-digamma-ineq-1}
\psi_{p,q}(s+t)\geq \psi_{p,q}(s) + \psi_{p,q}(t).
\end{equation}
\end{theorem}

\begin{proof}
Let $\mu(t)=\psi_{p,q}(s+t) - \psi_{p,q}(s) - \psi_{p,q}(t)$. Then fixing $s$ we have,
\begin{align*}
\mu'(t)=\psi'_{p,q}(s+t) - \psi'_{p,q}(t)&=  (\ln q)^{2}  \sum_{n=1}^{p} \left[ \frac{ nq^{n(s+t)}}{1-q^{n}} - \frac{ nq^{nt}}{1-q^{n}} \right] \\
&=   (\ln q)^{2}  \sum_{n=1}^{p} \frac{ nq^{nt}(q^{ns}-1)}{1-q^{n}} \leq0.
\end{align*}
That implies $\mu$ is non-increasing. Furthermore,

\begin{align*}
\lim_{t \rightarrow \infty}\mu(t)&= \lim_{t \rightarrow \infty} \left[\psi_{p,q}(s+t) - \psi_{p,q}(s) - \psi_{p,q}(t) \right] \\
&= - \ln[p]_q +(\ln q)\lim_{t \rightarrow \infty}\sum_{n=1}^{p} \left[  \frac{ q^{n(s+t)}}{1-q^{n}} - 
\frac{ q^{ns}}{1-q^{n}} -  \frac{ q^{nt}}{1-q^{n}}    \right] \\
&= - \ln[p]_q +(\ln q)\lim_{t \rightarrow \infty}\sum_{n=1}^{p} \left[ \frac{ q^{ns}.q^{nt}-q^{ns}-q^{nt}}{1-q^{n}}\right]\\ 
&= - \ln[p]_q  - (\ln q) \sum_{n=1}^{p} \frac{ q^{ns}}{1-q^{n}}\geq0.
\end{align*}
Therefore $\mu(t)\geq0$ concluding the proof.
\end{proof}

\begin{theorem}\label{thm:(p,q)-polygamma-ineq-1}
Let $s, t>0$,  $ q\in(0,1)$ and $p \in N$. Suppose that $m$ is a positive odd integer, then the following inequality is valid.
\begin{equation}\label{eqn:(p,q)-polygamma-ineq-1}
\psi^{(m)}_{p,q}(s+t)\leq \psi^{(m)}_{p,q}(s) + \psi^{(m)}_{p,q}(t) .
\end{equation}
\end{theorem}

\begin{proof}
Let $\eta(t)=\psi^{(m)}_{p,q}(s+t) - \psi^{(m)}_{p,q}(s) - \psi^{(m)}_{p,q}(t)$. Then fixing $s$ we have,
\begin{align*}
\eta'(t)&=\psi^{(m+1)}_{p,q}(s+t) - \psi^{(m+1)}_{p,q}(t)\\
&=  (\ln q)^{m+2}  \sum_{n=1}^{p} \left[ \frac{ n^{m+1}q^{n(s+t)}}{1-q^{n}} - \frac{ n^{m+1}q^{nt}}{1-q^{n}} \right] \\
&=  (\ln q)^{m+2}  \sum_{n=1}^{p} \left[ \frac{ n^{m+1}q^{nt}(q^{ns}-1)}{1-q^{n}} \right]\geq0. \,\, \text{(since $m$ is odd)} 
\end{align*}
\noindent
That implies $\eta$ is non-decreasing. Furthermore,

\begin{align*}
\lim_{t \rightarrow \infty}\eta(t)&= (\ln q)^{m+1}\lim_{t \rightarrow \infty}\sum_{n=1}^{p}  \left[\frac{ n^{m}q^{n(s+t)}}{1-q^{n}} - \frac{ n^{m}q^{ns}}{1-q^{n}} - \frac{ n^{m}q^{nt}}{1-q^{n}} \right] \\
&= (\ln q)^{m+1}\lim_{t \rightarrow \infty}\sum_{n=1}^{p}  \left[\frac{ n^{m}q^{ns}.q^{nt}}{1-q^{n}} - \frac{ n^{m}q^{ns}}{1-q^{n}} - \frac{ n^{m}q^{nt}}{1-q^{n}} \right] \\
&= -(\ln q)^{m+1} \sum_{n=1}^{p}\frac{ n^{m}q^{ns}}{1-q^{n}}\leq0. \, \, \text{(since $m$ is odd)}
\end{align*}
Therefore $\eta(t)\leq0$ concluding the proof.
\end{proof}

\begin{theorem}\label{thm:(p,q)-polygamma-ineq-2}
Let $s, t>0$,  $ q\in(0,1)$ and $p \in N$. Suppose that $m$ is a positive even integer, then the following inequality is valid.
\begin{equation}\label{eqn:(p,q)-polygamma-ineq-2}
\psi^{(m)}_{p,q}(s+t)\geq \psi^{(m)}_{p,q}(s) + \psi^{(m)}_{p,q}(t) .
\end{equation}
\end{theorem}

\begin{proof}
Let $\lambda(t)=\psi^{(m)}_{p,q}(s+t) - \psi^{(m)}_{p,q}(s) - \psi^{(m)}_{p,q}(t)$. Then fixing $s$ we have,
\begin{align*}
\lambda'(t)&=\psi^{(m+1)}_{p,q}(s+t) - \psi^{(m+1)}_{p,q}(t)\\
&=  (\ln q)^{m+2}  \sum_{n=1}^{p} \left[ \frac{ n^{m+1}q^{n(s+t)}}{1-q^{n}} - \frac{ n^{m+1}q^{nt}}{1-q^{n}} \right] \\
&=  (\ln q)^{m+2}  \sum_{n=1}^{p} \left[ \frac{ n^{m+1}q^{nt}(q^{ns}-1)}{1-q^{n}} \right]\leq0. \,\, \text{(since $m$ is even)} 
\end{align*}
\noindent
That implies $\lambda$ is non-decreasing. Furthermore,

\begin{align*}
\lim_{t \rightarrow \infty}\lambda(t)&= (\ln q)^{m+1}\lim_{t \rightarrow \infty}\sum_{n=1}^{p}  \left[\frac{ n^{m}q^{n(s+t)}}{1-q^{n}} - \frac{ n^{m}q^{ns}}{1-q^{n}} - \frac{ n^{m}q^{nt}}{1-q^{n}} \right] \\
&= -(\ln q)^{m+1} \sum_{n=1}^{p}\frac{ n^{m}q^{ns}}{1-q^{n}}\geq0. \, \, \text{(since $m$ is even)}
\end{align*}
Therefore $\lambda(t)\geq0$ concluding the proof.
\end{proof}

\begin{theorem}\label{thm:(p,q)-polygamma-ineq-3}
Let $s, t>0$,  $ q\in(0,1)$ and $p\in N$. Suppose $m$ is a positive odd integer, then the following inequality holds true.
\begin{equation}\label{eqn:(p,q)-polygamma-ineq-3}
\psi_{p,q}^{(m)}(s)\psi_{p,q}^{(m)}(t)\geq \left[ \psi_{p,q}^{(m)}(s+t) \right]^2
\end{equation}
\end{theorem}

\begin{proof}
We proceed as follows.
\begin{align*}
\psi_{p,q}^{(m)}(s) - \psi_{p,q}^{(m)}(s+t)&= (\ln q)^{m+1}\sum_{n=1}^{p} \left[  \frac{n^{m}q^{ns}}{1-q^{n}} - \frac{n^{m}q^{n(s+t)}}{1-q^{n}} \right] \\
&=(\ln q)^{m+1}\sum_{n=1}^{p} \left[  \frac{n^{m}q^{ns} (1- q^{nt})}{1-q^{n}} \right]\geq0.  \, \text{(since $m$ is odd)}
\end{align*}
That implies,
\begin{equation*}
\psi_{p,q}^{(m)}(s) \geq \psi_{p,q}^{(m)}(s+t) \geq0.
\end{equation*}
Similarly we have,
\begin{equation*}
\psi_{p,q}^{(m)}(t) \geq \psi_{p,q}^{(m)}(s+t) \geq0.
\end{equation*}
Multiplying these inequalities yields the desired results. Thus,
\begin{equation*}
\psi_{p,q}^{(m)}(s)\psi_{p,q}^{(m)}(t)\geq \left[ \psi_{p,q}^{(m)}(s+t) \right]^2 .
\end{equation*}
\end{proof}

\section{Concluding Remarks}
\begin{remark}
If in inequalities ~(\ref{eqn:(p,q)-digamma-ineq-1}), ~(\ref{eqn:(p,q)-polygamma-ineq-1}), ~(\ref{eqn:(p,q)-polygamma-ineq-2}) and ~(\ref{eqn:(p,q)-polygamma-ineq-3})  we allow  $p\rightarrow \infty$ as  $q\rightarrow 1^{-}$, then we repectively recover the inequalities  ~(\ref{eqn:digamma-ineq}), ~(\ref{eqn:polygamma-ineq-1}), ~(\ref{eqn:polygamma-ineq-2}) and ~(\ref{eqn:polygamma-ineq-3}). We have thus generalized the earlier results as in  \cite{Mansour-Shabani-2009} and \cite{Sulaiman-2011}.  The $k$, $p$ and $q$ analogues of  ~(\ref{eqn:digamma-ineq}), ~(\ref{eqn:polygamma-ineq-1}) and ~(\ref{eqn:polygamma-ineq-2}) can be found in the papers \cite{Nantomah-Prempeh-2014a}, \cite{Nantomah-Prempeh-2014b} and \cite{Nantomah-Prempeh-2014c}. Also, the $(q,k)$-analogues of  ~(\ref{eqn:(p,q)-digamma-ineq-1}), ~(\ref{eqn:(p,q)-polygamma-ineq-1}), ~(\ref{eqn:(p,q)-polygamma-ineq-2}) and ~(\ref{eqn:(p,q)-polygamma-ineq-3}) can be found in \cite{Kwara-Nantomah-2014}. \\
\end{remark}

\subsection*{Acknowledgments}
 The authors would like to thank the anonymous referee for his/her comments that helped us improve this article.

\bibliographystyle{plain}


\end{document}